\documentclass[a4paper,10pt]{amsart}
\usepackage[english]{babel}
\usepackage{amsmath,amssymb,amsthm,amscd}

\usepackage{graphicx}

\theoremstyle{plain}
\newtheorem{theorem}{Theorem}

\newtheorem{lemma}[theorem]{Lemma}

\theoremstyle{definition}

\newcommand{\N}{\ensuremath{\mathbb{N}}}

\newcommand{\R}{\ensuremath{\mathbb{R}}}

\newcommand{\ex}{\text{ext}}

\DeclareMathOperator{\e}{e}
\newcommand{\ext}{\operatorname{ext}}

\renewcommand{\leq}{\leqslant}
\renewcommand{\geq}{\geqslant}

\begin{document}
\title[Numerical index of absolute symmetric norms]{On the numerical index of absolute symmetric norms on the plane}
 \thanks{Research partially supported by projects PGC2018-093794-B-I00 (MCIU/AEI/FEDER, UE) and FQM-185 (Junta de Andaluc\'{i}a/FEDER, UE). The second author is also supported by the Ph.D. scholarship FPU18/03057 (MECD)}
 \date{December 23rd, 2019}

\maketitle

\centerline{\textsc{\large Javier
		Mer\'{\i}}  \footnote{Corresponding
author. \emph{E-mail:} \texttt{jmeri@ugr.es}} \quad and \quad \textsc{\large Alicia Quero}}

\begin{center} Departamento de An\'{a}lisis Matem\'{a}tico \\ Facultad de
Ciencias \\ Universidad de Granada \\ 18071 Granada, SPAIN \\
\emph{E-mail addresses:} \texttt{jmeri@ugr.es}, \ \texttt{aliciaquero@ugr.es}
\end{center}

\thispagestyle{empty}

\begin{abstract}
We give a lower bound for the numerical index of two-dimensional real spaces with absolute and symmetric norm. This allows us to compute the numerical index of the two-dimensional real $L_p$-space for $3/2\leq p\leq 3$.
\end{abstract}

\maketitle

\section[\S 1. Introduction]{Introduction}
The numerical index of a Banach space is a constant relating the
norm and the numerical range of bounded linear operators on
the space. Let us recall the relevant definitions.
Given a Banach space $X$, we will write $X^*$ for its
topological dual and $\mathcal{L}(X)$ for the Banach algebra of
all (bounded linear) operators on $X$. For an operator $T\in
\mathcal{L}(X)$, its \emph{numerical range} is defined as
$$
V(T):=\{x^*(Tx) \colon x^*\in X^*,\ x\in X,\ \|x^*\|=\|x\|=x^*(x)=1 \},
$$
and its \emph{numerical radius} is
$$
v(T):=\sup\{|\lambda| \colon \lambda\in V(T) \}.
$$
Clearly, $v$ is a seminorm on $\mathcal{L}(X)$
satisfying $v(T)\leq \|T\|$ for every $T\in \mathcal{L}(X)$. The \emph{numerical index} of
$X$ is the constant given by
$$
n(X):=\inf\{v(T) \colon T\in \mathcal{L}(X),\ \|T\|=1\}
$$
or, equivalently, $n(X)$ is the greatest constant $k\geq 0$ satisfying
$k\,\|T\| \leq v(T)$ for every $T\in \mathcal{L}(X)$.
Classical references on numerical index are the paper
\cite{D-Mc-P-W} and the monographs by F.F.~Bonsall and J.~Duncan
\cite{B-D1,B-D2} from the seventies. There has been a deep development of this field of study with the contribution of several authors. The reader will find the
state of the art on the subject in the survey paper
\cite{KaMaPa} and references therein.

In the following we recall some results concerning the numerical index which
will be relevant to our discussion. It is clear that $0\leq
n(X)\leq 1$ for every Banach space $X$. In the real case, all values in $[0,1]$ are
possible for the numerical index. In the complex case, one has
$1/\e\leq n(X)\leq 1$ and all of these values are possible. Let us
also mention that $v(T^*)=v(T)$ for every
$T\in \mathcal{L}(X)$, where $T^*$ is the adjoint operator of $T$ (see
\cite[\S~9]{B-D1}), so it clearly follows that $n(X^*)\leqslant
n(X)$. Although the equality does not
always hold, when $X$ is a reflexive space, one clearly gets
$n(X)=n(X^*)$. There are some classical Banach spaces for
which the numerical index has been calculated. If $H$ is a Hilbert space of dimension
greater than one, then $n(H)=0$ in the real case and $n(H)=1/2$ in
the complex case. Besides, $n(L_1(\mu))=1$ and the same happens to all 
its isometric preduals. In particular, it follows that $n\bigl(C(K)\bigr)=1$
for every compact $K$. 

The problem of computing the numerical index of the
$L_p$-spaces has been latent since the beginning of the theory \cite{D-Mc-P-W}. In order to present the known results on this matter we need to fix some notation. For $1< p<\infty$, 
we write $\ell_p^m$ for the $m$-dimensional
$L_p$-space, $q=p/(p-1)$
for the conjugate exponent to $p$, and
$$
M_p:=\max_{t\in[0,1]} \frac{|t^{p-1}-t|}{1+t^p}=\max_{t\geq 1} \frac{|t^{p-1}-t|}{1+t^p},
$$
which is the numerical radius of the operator represented by the matrix $\begin{pmatrix}0 & 1 \\ -1 & 0 \end{pmatrix}$
defined on the real space $\ell_p^2$. This can be found in \cite[Lemma~2]{MarMer-LP}, where it is also observed that $M_q=M_p$.
Although it is known that
$\bigl\{n(\ell_p^{2})\ : \ 1< p < \infty\bigr\}=[0,1[$ in the
real case (see \cite[p.~488]{D-Mc-P-W}), the exact computation of $n(\ell_p^2)$
has not been achieved for $p\neq 2$, all the more of $n(\ell_p)$. However, some
results have been obtained on the numerical index of the
$L_p$-spaces \cite{Eddari,Eddari2,Eddari3,MarMer-LP,MMP-Israel}, we summarize them in the following list. 
\begin{itemize}
	\item[(a)] The sequence $\bigl(n(\ell_p^m)\bigr)_{m\in\N}$
	is decreasing.
	\item[(b)] $n\bigl(L_p(\mu)\bigr)=\inf \{n(\ell_p^m)\,:\,
	m\in\N\}$ for every measure $\mu$ such that
	$\dim\bigl(L_p(\mu)\bigr)=\infty$.
	\item[(c)] In the real case, $n(L_p[0,1])\geq \frac{M_p}{12}$.
	\item[(d)] In the real case, $\displaystyle
	\max\left\{\frac{1}{2^{1/p}},\
	\frac{1}{2^{1/q}}\right\}\,M_p\leq n(\ell_p^{2})\leq
	M_p$.
\end{itemize}
The presence of the numerical radius of the operator represented by the matrix $\begin{pmatrix}0 & 1 \\ -1 & 0 \end{pmatrix}$ in the value of the numerical index of $L_p$-spaces is not a coincidence. Although there are not too many examples of Banach spaces for which the numerical index has been computed, for those two-dimensional real spaces with absolute and symmetric norm whose numerical index is known, it coincides with the numerical radius of the mentioned operator. This happens, for instance, to a family of octagonal norms and to the spaces whose unit ball is a regular polygon, see \cite[Theorem~2 and Theorem~5]{MartinMeri-LAMA}. The aim of this paper is to show that the same happens for many absolute and symmetric norms on $\R^2$, this is the content of Theorem~\ref{thm:num-index-abs-norms}. We say that a norm $\|\cdot\|: \R^2 \longrightarrow \R$ is \emph{absolute} if $\|(1,0)\|=\|(0,1)\|=1$ and 
$$
\|(a,b)\|=\|(|a|,|b|)\|
$$
for every $a,b \in \R$, and that the norm is \emph{symmetric} if $\|(b,a)\|=\|(a,b)\|$ for every $a,b \in \R$. 
Some of the most important examples of absolute and symmetric norms are $\ell_p$-norms on $\R^2$. As a major consequence of Theorem~\ref{thm:num-index-abs-norms} we show that $n(\ell_p^2)=M_p$ for $3/2\leq p\leq 3$, which improves partially \cite[Theorem~1]{MarMer-LP} and throws some light to the long standing problem of computing the numerical index of $L_p$-spaces.

To finish the introduction, we recall some facts about
numerical radius and about optimization of linear functions on convex sets that will be useful in our arguments. Let $X$ be a Banach space, and suppose that 
$S\in \mathcal{L}(X)$ is an onto isometry. Then, for
every operator $T\in \mathcal{L}(X)$, it is easy to check that
$$
v(T)=v(\pm S^{-1}TS).
$$
This becomes particularly useful when $X$ is $\R^2$ endowed with an absolute and symmetric norm, as we can find a basis of the space of operators $\mathcal{L}(X)$ formed by onto isometries:
$$ 
I_1=\begin{pmatrix}
1 & 0\\
0 & 1
\end{pmatrix}, \qquad
I_2=\begin{pmatrix}
1 & 0\\
0 & -1
\end{pmatrix}, \qquad
I_3=\begin{pmatrix}
0 & 1\\
1 & 0
\end{pmatrix}, \qquad
I_4=\begin{pmatrix}
0 & 1\\
-1 & 0
\end{pmatrix}\,.
$$  

For a convex set $A$, $\ex(A)$ stands for the set of its
extreme points, that is, those points which are not the mid point of
any non-trivial segment contained in $A$. By Minkowski's Theorem (see \cite[Corollary~1.13]{Tuy} for
instance) a nonempty compact convex subset of
$\R^n$ is equal to the convex hull of its extreme points. Therefore, every linear function on a compact convex set attains its minimum (and its maximum) at an extreme point of the set. 

\section{The results}
We start with an easy lemma showing that, for two dimensional real spaces with absolute and symmetric norm, the elements in the numerical range of $I_4$ are smaller than those of $I_j$ for $j=1,2,3$. 

\begin{lemma}\label{lemma-c4}
	Let $X$ be $\R^2$ endowed with an absolute and symmetric norm. Then
	$$ |x^*(I_jx)|\geq |x^*(I_4x)| \qquad (j=1,2,3)$$
	for every $x\in S_X$ and $x^*\in S_{X^*}$ such that $x^*(x)=1$.
\end{lemma}
\begin{proof}
	Fixed $x=(a,b)\in S_X$ and $x^*=(\alpha,\beta)\in S_{X^*}$ with $x^*(x)=\alpha a +\beta b=1$, it is obvious that $1=|x^*(I_1x)|\geq|x^*(I_4x)|$. To prove $|x^*(I_3x)|\geq|x^*(I_4x)|$ observe that	
	\begin{align}\label{eq:duality-equality}
	1=\alpha a +\beta b&\leq |\alpha|\,|a|+|\beta|\,|b|\leq \|(|\alpha|,|\beta|)\| \, \|(|a|,|b|)\|=\|x^*\| \, \|x\|=1
	\end{align}
	which clearly implies $\alpha a=|\alpha|\,|a|$ and $\beta b=|\beta|\,|b|$. Moreover, we deduce that $\alpha b$ and $\beta a$ have the same sign as $\alpha a \, \beta b\geq 0$ and, therefore, 
	$$
	|x^*(I_3x)|=|\alpha b+ \beta a|=|\alpha|\,|b|+|\beta| \,|a|\geq |\alpha b- \beta a|=|x^*(I_4x)|.
	$$
	To prove $|x^*(I_2x)|\geq|x^*(I_4x)|$ observe that 
	\begin{align*}
	|x^*(I_2x)|&=|\alpha a-\beta b|=\big||\alpha|\,|a|-|\beta|\,|b|\big| \qquad \text{and} \qquad |x^*(I_4x)|=|\alpha b-\beta a|=\big||\alpha|\,|b|-|\beta|\,|a|\big|.
	\end{align*} 
	So, when $|a|=|b|$, it is evident that $|x^*(I_2x)|= |x^*(I_4x)|$. When $|a|\neq |b|$ we need the following claim.
	
	\emph{Claim:} $|a|>|b|$ implies $|\alpha|\geq|\beta|$ and $|b|>|a|$ implies $|\beta|\geq|\alpha|$. 
	
	We only show the first implication, as the second one is analogous. Using the symmetry of the norm and \eqref{eq:duality-equality} we can write
		\begin{align*}
		\|(|\beta|,|\alpha|)\| \, \|(|a|,|b|)\|= \|(|\alpha|,|\beta|)\| \, \|(|a|,|b|)\|= |\alpha|\,|a|+|\beta|\, |b|.
		\end{align*} 
		On the other hand, writing $y^*=(|\beta|,|\alpha|)$ and $y=(|a|,|b|)$,it is clear that
		$$ 
		\|(|\beta|,|\alpha|)\| \, \|(|a|,|b|)\|\geq y^*(y)= |\beta|\,|a|+|\alpha|\,|b|. 
		$$
		Therefore, we get $|\beta|\,|a|+|\alpha|\,|b|\leq |\alpha|\,|a|+|\beta|\, |b| $, and so
		$|\beta|(|a|-|b|)\leq|\alpha|(|a|-|b|)$. Since $|a|>|b|$, it follows that $|\alpha|\geq |\beta|$ and the claim is proved. 
		 
	Let us finish the proof of $|x^*(I_2x)|\geq|x^*(I_4x)|$. If $|a|>|b|$, we get $|\alpha|\geq|\beta|$ by the claim and, moreover, $|\alpha|\,|a|\geq |\alpha|\,|b|\geq |\beta| \, |b|$ and $|\alpha|\,|a|\geq |\beta|\,|a|\geq |\beta| \, |b|$ hold, which clearly imply
	$$
	|x^*(I_2x)|=\big||\alpha|\,|a|-|\beta|\,|b|\big|\geq \big||\alpha|\,|b|-|\beta|\,|a|\big|=|x^*(I_4x)|.
	$$
	The remaining case $|b|>|a|$ is completely analogous.	
\end{proof}

We are ready to state and prove the first main result of the paper.

\begin{theorem}\label{thm:num-index-abs-norms}
	Let $X$ be $\R^2$ endowed with an absolute and symmetric norm. Let $x_0\in S_X$ and $x_0^*\in S_{X^*}$ be such that $|x_0^*(I_4 x_0)|=v(I_4)$ and write $c_j=|x_0^*(I_jx_0)|$ for every $j=1,\ldots,4$. If $c_4=0$, then $n(X)=0$. If otherwise $c_4>0$, then
	$$
	n(X)\geq\min\left\{c_4,\dfrac{2}{1+\frac{1}{c_2}+\frac{1}{c_3}+\frac{1}{c_4}} \right\}.
	$$
	Moreover, if the inequality $c_4\left(1+\frac{1}{c_2}+\frac{1}{c_3}\right)\leq 1$ holds, then
	$$
	n(X)=v(I_4).
	$$
\end{theorem}

\begin{proof}
	Observe first that $n(X)\leq v(I_4)$ since $\|I_4\|=1$. So $n(X)=0$ holds when $c_4=0$. Thus we assume that $c_4>0$ which, by Lemma~\ref{lemma-c4}, implies $c_j>0$ for $j=2,3$.
	
Fixed a non-zero operator $T\in \mathcal{L}(X)$ our aim is to estimate $\frac{v(T)}{\|T\|}$. To do so, observe that 
there exist $A_j\in \R$ for $j=1,\ldots,4$ satisfying $T=\sum_{k=1}^{4} A_k I_k$, as the onto isometries $I_1,\ldots,I_4$ form a basis of $\mathcal{L}(X)$. Observe next that 
\begin{align*}
I_1^{-1} T I_1&=A_1I_1+A_2I_2+A_3I_3+A_4I_4\\
I_2^{-1} T I_2&=A_1I_1+A_2I_2-A_3I_3-A_4I_4\\
I_3^{-1} T I_3&=A_1I_1-A_2I_2+A_3I_3-A_4I_4\\
I_4^{-1} T I_4&=A_1I_1-A_2I_2-A_3I_3+A_4I_4\\
\end{align*}
so, using that $v(T)=v(\pm I_j^{-1}TI_j)$ for every $j=1,\ldots,4$, we can write
	\begin{align*}
	v(T)=\max\left\{\right. &\left|\pm v(I_j^{-1} T I_j)\right|\colon j=1,\ldots,4 \left. \right\} \\
	\geq \max\left\{\right. &\left|\pm x_0^*(I_j^{-1} T I_j x_0)\right|\colon j=1,\ldots,4 \left. \right\} \\
	= \max \left\{\right.&\left|\pm (A_1x_0^*(I_1x_0)+A_2x_0^*(I_2x_0)+A_3x_0^*(I_3x_0)+A_4x_0^*(I_4x_0))\right|, \\
	&\left|\pm (A_1x_0^*(I_1x_0)+A_2x_0^*(I_2x_0)-A_3x_0^*(I_3x_0)-A_4x_0^*(I_4x_0))\right|, \\
	&\left|\pm (A_1x_0^*(I_1x_0)-A_2x_0^*(I_2x_0)+A_3x_0^*(I_3x_0)-A_4x_0^*(I_4x_0))\right|, \\
	&\left|\pm (A_1x_0^*(I_1x_0)-A_2x_0^*(I_2x_0)-A_3x_0^*(I_3x_0)+A_4x_0^*(I_4x_0))\right|\left.\right\}.
	\end{align*}
The combination of signs in the last expression allows us to deduce
$$
v(T) \geq \max\left\{\sum_{\substack{k=1\\k\neq j}}^{4}|A_k|c_k-|A_j|c_j\colon j=1,\ldots,4 \right\}.
$$
Now, writing $\|T\|_+=\sum_{k=1}^4|A_k|$, we get $\|T\|=\left\|\sum_{k=1}^4 A_kI_k\right\|\leq \|T\|_+$. Besides, calling $\alpha_j=\frac{|A_j|}{\|T\|_+}$ for $j=1, \ldots,4$, we can estimate $n(X)$ as follows: 
	\begin{align*}
	n(X)&=\inf \left\{\dfrac{v(T)}{\|T\|}\colon T\in \mathcal{L}(X), \ T\neq 0\right\} \geq \inf\left\{\dfrac{v(T)}{\|T\|_+}\colon T\in \mathcal{L}(X), \ T\neq 0\right\} \\ &\geq \min_{\substack{\alpha_1+\alpha_2+\alpha_3+\alpha_4=1\\\alpha_j\geq0}} \max \left\{\sum_{\substack{k=1\\k\neq j}}^4\alpha_k c_k -\alpha_j c_j\colon j=1,\ldots,4 \right\}.
	\end{align*} 
So, defining the function  
	$$
	f(\alpha_1,\alpha_2,\alpha_3,\alpha_4)=\max\left\{\sum_{\substack{k=1\\k\neq j}}^4\alpha_k c_k -\alpha_j c_j\colon j=1,\ldots,4 \right\}  \qquad \big((\alpha_1,\alpha_2,\alpha_3,\alpha_4)\in\R^4\big)
	$$ 
	and the compact set 
	$$
	K=\left\{(\alpha_1,\alpha_2,\alpha_3,\alpha_4)\in\R^4 \colon \sum_{k=1}^4\alpha_k=1, \alpha_j\geq0, j=1,\ldots,4 \right\},
	$$ 
	we have that 
	\begin{equation*}
	n(X)\geq \underset{K}{\min} f.
	\end{equation*}
	Our goal now is to compute this minimum. As $f$ is the maximum of linear functions, following a typical strategy of linear programming, we can transform this minimization problem into a linear optimization one: we have to minimize the function 
	$$
	g(\alpha_1,\alpha_2,\alpha_3,\alpha_4,z)=z \qquad \big((\alpha_1,\alpha_2,\alpha_3,\alpha_4,z)\in \R^5\big)
	$$ 
	on the compact convex set 
	$$
	K'=\left\{(\alpha_1,\alpha_2,\alpha_3,\alpha_4,z)\in\R^5 \colon \sum_{k=1}^4\alpha_k=1, z\leq2, \alpha_j\geq0,   z\geq\sum_{\substack{k=1\\k\neq j}}^4\alpha_k c_k -\alpha_j c_j, j=1,\ldots,4 \right\}.
	$$
	In fact, it is easy to check that
	$$
	\underset{K}{\min}\, f=\underset{K'}{\min}\, g.
	$$ 
	Indeed, if $(\alpha_1,\alpha_2,\alpha_3,\alpha_4)\in K$ is such that $\underset{K}{\min}\, f=f(\alpha_1,\alpha_2,\alpha_3,\alpha_4)$, then we clearly have that 
	$$
	\big(\alpha_1,\alpha_2,\alpha_3,\alpha_4, f(\alpha_1,\alpha_2,\alpha_3,\alpha_4)\big)\in K' \qquad \text{and} \qquad  g\big(\alpha_1,\alpha_2,\alpha_3,\alpha_4, f(\alpha_1,\alpha_2,\alpha_3,\alpha_4)\big)= f(\alpha_1,\alpha_2,\alpha_3,\alpha_4).
	$$
	Therefore, we have $\underset{K}{\min}\, f\geq \underset{K'}{\min}\, g$. To prove the reverse inequality take $(\alpha_1,\alpha_2,\alpha_3,\alpha_4,z)\in K'$ satisfying $\underset{K'}{\min}\, g=g(\alpha_1,\alpha_2,\alpha_3,\alpha_4,z)=z$ and observe that $(\alpha_1,\alpha_2,\alpha_3,\alpha_4)\in K$ and $f(\alpha_1,\alpha_2,\alpha_3,\alpha_4)\leq z$. So we get $\underset{K}{\min}\, f\leq \underset{K'}{\min}\, g$.
	
	To finish the proof we just have to compute $\underset{K'}{\min}\, g$. Since $K'$ is a compact convex set, the linear function $g$ attains its minimum on $K'$ at an extreme point of $K'$. Fixed $(\alpha_1,\alpha_2,\alpha_3,\alpha_4,z)\in \ext(K')$, as $K'\subset \R^5$, it must happen that at least five of the ten restrictions that define $K'$ become equalities. We calculate $g(\alpha_1,\alpha_2,\alpha_3,\alpha_4,z)$ depending on which equalities occur. If there exists $j_0\in\{1,\ldots,4\}$ such that $\alpha_{j_0}=0$, then
	$$
	g(\alpha_1,\alpha_2,\alpha_3,\alpha_4,z)=z\geq \sum_{\substack{k=1\\k\neq j_0}}\alpha_k c_k\geq c_4\sum_{\substack{k=1\\k\neq j_0}}\alpha_k=c_4,
	$$
	where we have used that $c_j\geq c_4$ for every $j\in\{1,2,3\}$ by Lemma~\ref{lemma-c4}.
	
	If otherwise $\alpha_j>0$ for every $j\in\{1,\ldots,4\}$, we have that $z=\sum_{\substack{k=1\\k\neq j}}^4\alpha_k c_k -\alpha_j c_j$ for every $j\in\{1,\ldots,4\}$, as $z<2$ whenever $z=\sum_{\substack{k=1\\k\neq j}}^4\alpha_k c_k -\alpha_j c_j$ for any $j$. Hence
	$$\sum_{\substack{k=2}}^4\alpha_k c_k -\alpha_1 c_1=\sum_{\substack{k=1\\k\neq 2}}^4\alpha_k c_k -\alpha_2 c_2=\sum_{\substack{k=1\\k\neq 3}}^4\alpha_k c_k -\alpha_3 c_3=\sum_{k=1}^3\alpha_k c_k -\alpha_4 c_4,$$
	and so $\alpha_1 c_1=\alpha_2 c_2=\alpha_3 c_3=\alpha_4 c_4$. Since $c_1=1$, we get
	$$ \alpha_2=\dfrac{\alpha_1}{c_2}, \quad \alpha_3=\dfrac{\alpha_1}{c_3}, \quad \alpha_4=\dfrac{\alpha_1}{c_4} $$ and it follows from $\alpha_1+\alpha_2+\alpha_3+\alpha_4=1$ that 
	$$
	\alpha_1=\dfrac{1}{1+\frac{1}{c_2}+\frac{1}{c_3}+\frac{1}{c_4}}.
	$$
	Therefore,
	$$
	g(\alpha_1,\alpha_2,\alpha_3,\alpha_4,z)=z=2\alpha_1=\dfrac{2}{1+\frac{1}{c_2}+\frac{1}{c_3}+\frac{1}{c_4}}.
	$$
	So, for every $(\alpha_1,\alpha_2,\alpha_3,\alpha_4,z)\in \ext(K')$ we have shown that either 
	$g(\alpha_1,\alpha_2,\alpha_3,\alpha_4,z)\geq \dfrac{2}{1+\frac{1}{c_2}+\frac{1}{c_3}+\frac{1}{c_4}}$ or $g(\alpha_1,\alpha_2,\alpha_3,\alpha_4,z)\geq c_4$. Thus, we can write
	$$
	n(X)\geq \min_{K'} g\geq \min\left\{c_4,\dfrac{2}{1+\frac{1}{c_2}+\frac{1}{c_3}+\frac{1}{c_4}}\right\}
	$$
	which finishes the first part of the proof. Finally, to prove the moreover part, it suffices to observe that if $c_4\left(1+\frac{1}{c_2}+\frac{1}{c_3}\right)\leq 1$ then $c_4\leq \dfrac{2}{1+\frac{1}{c_2}+\frac{1}{c_3}+\frac{1}{c_4}}$ and hence we get $n(X)=c_4=v(I_4)$.
\end{proof}

Using the preceding result we can obtain the numerical index of two-dimensional $L_p$-spaces for some values of $p$.
In order to use Theorem~\ref{thm:num-index-abs-norms} we need to find one pair $x\in S_{\ell_p^2}$, $x^*\in S_{\ell_q^2}$ satisfying $x^*(x)=1$, at which $I_4$ attains its numerical radius. However, this seems to be a rather tricky problem for arbitrary $p$. We can avoid this by showing that  condition $c_4\left(1+\frac{1}{c_2}+\frac{1}{c_3}\right)\leq 1$ in the statement of Theorem~\ref{thm:num-index-abs-norms} holds not only for a particular choice of $x\in S_{\ell_p^2}$, $x^*\in S_{\ell_q^*}$ satisfying $x^*(x)=1$ but for all of them.

\begin{theorem}
	Let $p\in\left[\frac{3}{2},3\right]$. Then 
	$$
	n(\ell_p^2)=M_p=\sup_{t\in[0,1]} \dfrac{|t^{p-1}-t|}{1+t^p}.
	$$
\end{theorem}
\begin{proof}
	It is known that $n(\ell_2^2)=0$. Besides, for $p\in ]2,3]$ we get $q\in [3/2, 2[$ so, using the fact that $n(\ell_p^2)=n(\ell_q^2)$, the result will be proved if we compute $n(\ell_p^2)$ for $p\in [3/2, 2[$. So we fix $p\in [3/2, 2[$ and we use the parametrization of the duality mapping for absolute norms on $\R^2$ given in \cite[Lemma~3.2]{D-Mc-P-W}. Indeed, for $t\in [0,1]$ consider 
	$$
	x_t=\frac{1}{(1+t^p)^{1/p}}(1,t)\qquad  \text{and} \qquad x^*_t=\frac{1}{(1+t^p)^{\frac{p-1}{p}}}(1,t^{p-1})
	$$ 
	which satisfy $x_t\in S_{\ell_p^2}$, $x_t^*\in S_{\ell_q^2}$, and  $x_t^*(x_t)=1$. We next define the functions 
	\begin{align*}
	c_1(t)&=x_t^*(I_1x_t)=1, \quad &c_2(t)&=x_t^*(I_2x_t)=\dfrac{1-t^p}{1+t^p}\,,& \\
	c_3(t)&=x_t^*(I_3x_t)=\dfrac{t^{p-1}+t}{1+t^p}, \quad &c_4(t)&=x_t^*(I_4x_t)=\dfrac{t^{p-1}-t}{1+t^p} &\qquad \big(t\in [0,1]\big).
	\end{align*}
	Since the maximum defining $v(I_4)=\max_{t\in[0,1]} \frac{t^{p-1}-t}{1+t^p}$ is obviously attained at some $t_0\in]0,1[$, if we show that $c_4(t)\left(1+\frac{1}{c_2(t)}+\frac{1}{c_3(t)}\right)\leq 1$ for every $t\in]0,1[$, then we will have $n(\ell_p^2)=v(I_4)=M_p$ by Theorem~\ref{thm:num-index-abs-norms}. So, for fixed $t\in]0,1[$, observe that
	\begin{align*}
	c_4(t)\left(1+\dfrac{1}{c_2(t)}+\dfrac{1}{c_3(t)}\right)&=\dfrac{t^{p-1}-t}{1+t^p}\left(1+\dfrac{1+t^p}{1-t^p}+\dfrac{1+t^p}{t^{p-1}+t}\right)=\dfrac{2t^{p-1}-2t}{1-t^{2p}}+\dfrac{t^{p-1}-t}{t^{p-1}+t}
	\end{align*}
	and, therefore
	\begin{align*}
	c_4(t)\left(1+\dfrac{1}{c_2(t)}+\dfrac{1}{c_3(t)}\right)\leq 1&\Longleftrightarrow \dfrac{2t^{p-1}-2t}{1-t^{2p}}+\dfrac{t^{p-1}-t}{t^{p-1}+t}\leq 1 \\ 
	&\Longleftrightarrow \dfrac{2(t^{p-1}-t)}{1-t^{2p}}\leq \dfrac{2t}{t^{p-1}+t}\\ 
	&\Longleftrightarrow 0\leq t-t^{2p-2}+t^2-t^{2p+1}\Longleftrightarrow 0\leq t(1-t^{2p-3})+t^2(1-t^{2p-1}).
	\end{align*}
	Since the last inequality holds for $3/2\leq p<2$ and $t\in]0,1[$, Theorem~\ref{thm:num-index-abs-norms} applies and finishes the proof.
\end{proof}

\end{document}